\pdfoutput=1
\documentclass{amsart}
\usepackage{amsmath}
\usepackage{amsthm}
\usepackage{amsfonts}
\usepackage{amssymb}
\usepackage{amscd}
\usepackage{mathtools}
\usepackage{multirow}
\usepackage{relsize}
\usepackage{algorithm}
\usepackage[noend]{algpseudocode}
\usepackage{enumitem}

\usepackage{graphicx}
\usepackage{color} 
\usepackage{transparent}

\theoremstyle{plain}
\newtheorem{theorem}{Theorem}[section]

\newtheorem{lemma}[theorem]{Lemma}
\newtheorem{proposition}[theorem]{Proposition}
\newtheorem{corollary}[theorem]{Corollary}
\newtheorem{conjecture}[theorem]{Conjecture}
\theoremstyle{definition}
\newtheorem{definition}[theorem]{Definition}

\makeatletter
\newcommand{\newreptheorem}[2]{\newtheorem*{rep@#1}{\rep@title}\newenvironment{‌​rep#1}[1]{\def\rep@title{#2 \ref*{##1}}\begin{rep@#1}}{\end{rep@#1}}}
\makeatother

\numberwithin{equation}{section}
\numberwithin{figure}{section}

\newcommand{\inv}{^{-1}}

\newcommand{\vol[1]}{\ensuremath \text{vol}( #1 )}

\begin{document}

\title{The Determinant and Volume of 2-Bridge Links and Alternating 3-Braids}
\author{Stephan D. Burton}
\thanks{Supported by NSF Grants DMS-1105843 and DMS-1404754.}

\maketitle

\begin{abstract}
We examine the conjecture, due to Champanerkar, Kofman, and Purcell \cite{CKP} that $\vol[K] < 2 \pi \log \det (K)$ for alternating hyperbolic links, where $\vol[K] = \vol[S^3\backslash K]$ is the hyperbolic volume and $\det(K)$ is the determinant of $K$. We prove that the conjecture holds for 2-bridge links, alternating 3-braids, and various other infinite families. We show the conjecture holds for highly twisted links and quantify this by showing the conjecture holds when the crossing number of $K$ exceeds some function of the twist number of $K$.
\end{abstract}


\section{Introduction}
A major goal in the study of knots is to relate combinatorial and topological properties of knots to the hyperbolic geometry of knots. In this paper, we explore the relationship between the hyperbolic volume $\vol[K] = \vol[S^3 \backslash K]$ of an alternating hyperbolic knot and its determinant $\det(K)$.

Dunfield noted a relationship between the volume and determinant of a knot in an online post \cite{DunfieldOnline}. He observed that there is a nearly linear relationship between the hyperbolic volume of an alternating knot and $\log(J(-1))$ where $J$ denotes the Jones polynomial. After further study of this relationship and some experimentation, Champanerkar, Kofman and Purcell \cite{CKP} made the following conjecture.

\begin{conjecture}\label{conj:det_vol}
Let $K$ be a hyperbolic alternating knot. Then $\emph{vol}(K) < 2 \pi \log \det(K)$.
\end{conjecture}
One can use the data from Knotscape  \cite{Knotscape} and SnapPy \cite{SnapPy} to verify this conjecture for all alternating knots with up to 16 crossings. Champanerkar, Kofman and Purcell in \cite{CKP} computationally verified Conjecture \ref{conj:det_vol} for many examples of an infinite family of links known as weaving knots. Using these weaving knots, they showed that the constant $2\pi$ is sharp, in the sense that given $\alpha < 2 \pi$ there exists an alternating link $K$ with $\alpha \log(\det(K)) < \vol[K]$. 

Stoimenow \cite{StoimenowGraphsDeterminants} also explored the relationship between volume and determinant, and showed that if $K$ is a non-trivial, non-split, alternating hyperbolic link then \begin{equation}
\det(K) \geq 2(1.0355)^{\vol[K]}
\end{equation} 
He further demonstrated that there exist constants $C_1, C_2 > 0$ such that for any hyperbolic link $K$
\begin{equation}
\det(K) \leq \left(\dfrac{C_1 c(K)}{\vol[K]} \right)^{C_2 \vol[K]}
\end{equation}
where $c(K)$ denotes the crossing number of $K$.

In this paper, we will verify that Conjecture \ref{conj:det_vol} holds for various infinite families of knots including 2-bridge knots and 3-braids. To obtain upper bounds on the volumes of knots we largely rely on work of Adams \cite{AdamsBipyramids} who gave an upper bound in terms of volumes of bipyramids. Adams \textit{et al.} \cite{AdamsEtAl} used this upper bound to study the volume densities of 2-bridge knots. Other useful upper bounds in the case of highly twisted knots are due to Lackenby, Agol and Thurston \cite{LackenbyBound} and Futer, Kalfagianni and Purcell \cite{Guts}.

To study the determinant of knots, we will count the number of spanning trees of a graph associated to the checkerboard coloring of a diagram of the knot. We rely on a recurrence equation due to Kauffman \cite{KauffmanRational} as well as two well-known combinatorial theorems for counting spanning trees. In the case of highly twisted knots, we utilize work of Stoimenow \cite{StoimenowGraphsDeterminants} who provided a lower bound on the number of spanning trees in certain graphs.

An outline of the paper is as follows. In Section \ref{sec:background}, we outline the technology used in this paper to estimate volumes and determinants of knots. In Section \ref{sec:two_bridge}, we will prove Conjecture \ref{conj:det_vol} for 2-bridge links. In Section \ref{sec:3_braids}, we will prove the conjecture for alternating 3-braids and an infinite family of 4-braids. We discuss a general result about highly twisted links in Section \ref{sec:pretzel} and include an application to alternating pretzel links.

\textbf{Acknowledgements:} The author would like to thank his adviser, Efstratia Kalfagianni, for suggesting to study Conjecture \ref{conj:det_vol}, and for helpful comments. The author is also thankful for helpful conversations with David Futer.

\section{Background}\label{sec:background}
In this section we discuss the relevant theorems that will be used in the rest of the paper. We begin with a discussion of how one may find upper bounds on volumes of alternating hyperbolic links, and conclude with results on how one may calculate the determinant.

\subsection{Hyperbolic Volumes}
Adams in \cite{AdamsBipyramids} developed a method for finding an upper bound of the volume of an alternating hyperbolic link given an alternating diagram of the link. We will recall his notation and results. 

\begin{definition}
The regular ideal $n$-bipyramid can be formed as follows. Begin with $n$ ideal tetrahedra each having dihedral angles $\frac{2\pi}{n}, \frac{(n-2)\pi}{2n}, \frac{(n-2)\pi}{2n}$. Let $e$ be an edge running from a point in $\partial \mathbb{H}^3$ to  $\infty$. Glue an edge of each ideal tetrahedron with dihedral angle $2\pi/n$ to the edge $e$. The resulting polyhedron is called a \textit{regular ideal $n$-bipyramid} and will be denoted by $B_n$.
\end{definition}
\noindent An example of a regular ideal $n$-bipyramid is shown in Figure \ref{fig:bipyramid}. The volume of $B_n$ is given by
\begin{equation}
\vol[B_n] = n \left( \int_0^{\frac{2\pi}{n}} -\ln|2\sin(\theta)|\, d \theta + 2 \int_0^{\frac{\pi(n-2)}{2n}} -\ln|2\sin(\theta)|\, d\theta \right).
\end{equation}

\begin{figure}
\includegraphics[scale=.5]{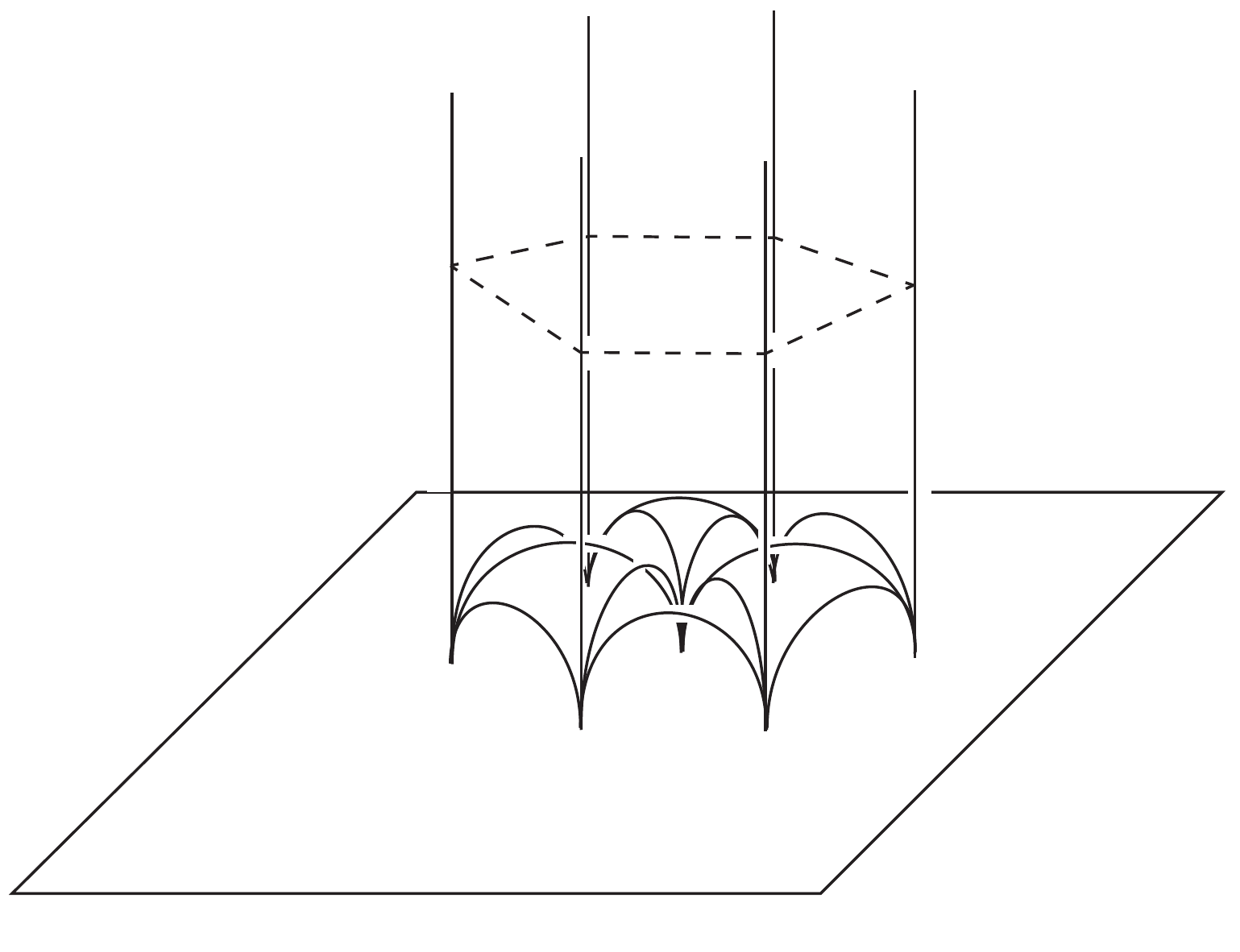}
\caption{A regular ideal 6-bipyramid. Figure taken from \cite{AdamsBipyramids}.} \label{fig:bipyramid}
\end{figure}

\noindent Adams \cite{AdamsBipyramids} proved the following theorem about the volumes of regular ideal $n$-bipyramids.

\begin{theorem}[\cite{AdamsBipyramids}]\label{thm:BipyramidVolume}
The volume of a regular ideal $n$-bipyramid satisfies the inequality $$\emph{vol}(B_n) < 2 \pi \log\left(\dfrac{n}{2}\right).$$ Moreover, this inequality is asymptotically sharp.
\end{theorem}

Adams used regular ideal $n-$bipryamids to give an upper bound on the volume of hyperbolic, alternating links. The following directly follows from \cite[Theorem 4.1]{AdamsBipyramids}.

\begin{theorem}\label{thm:AdamsUpperBound}
Let $K$ be a hyperbolic link with a reduced alternating projection $D$. Let $b_n$ be the number of faces of $D$ having $n$ edges. Suppose that there are two distinct faces of $D$ having respectively $r$ and $s$ edges. Then
\begin{equation}
\emph{vol}(K) \leq  - \emph{vol}(B_r) - \emph{vol}(B_s) + \sum b_n \emph{vol}(B_n)
\end{equation}
\end{theorem}

\noindent Combining Theorems \ref{thm:BipyramidVolume} and \ref{thm:AdamsUpperBound} we get the following corollary.

\begin{corollary}\label{cor:AdamsVolumeBound}
Let $K$ be a hyperbolic knot having an alternating projection $D$. Let $b_n$ be the number of faces of $D$ having $n$ edges. Suppose that there are two distinct faces of $D$ having respectively $r$ and $s$ edges. Then
\begin{equation}
\emph{vol}(K) < 2 \pi \log \left( \frac{\prod n^{b_n}}{2^m} \frac{4}{rs}\right)
\end{equation}
\end{corollary}
\begin{proof}
This is a straightforward calculation obtained by inserting the inequality of Theorem \ref{thm:BipyramidVolume} into Theorem \ref{thm:AdamsUpperBound}.
\begin{align*}
\vol[K] &\leq - \vol[B_r] - \vol[B_s] + \sum b_n \vol[B_n]\\
& = (b_r-1)\vol[B_r] + (b_s-1)\vol[B_s]  + \sum_{n \neq r, s} b_n \vol[B_n]\\
& < 2 \pi \left[ (b_r-1)\log\left(\dfrac{r}{2}\right) + (b_s-1) \log\left( \dfrac{s}{2}\right) + \sum_{n \neq r, s} b_n \log\left(\dfrac{n}{2}\right) \right]\\
& = 2 \pi \log \left(\dfrac{r^{b_r-1}}{2^{b_r-1}}\dfrac{s^{b_s-1}}{2^{b_s-1}}\prod_{n \neq r, s} \dfrac{n^{b_n}}{2^{b_n}}\right)\\
& = 2 \pi \log \left( \frac{\prod n^{b_n}}{2^m} \frac{4}{rs}\right)
\end{align*}
\end{proof}

The volume bound of Corollary \ref{cor:AdamsVolumeBound} is insufficient in certain cases involving links with a large number of crossings in a twist region. To handle this case, we appeal to the following theorem of Lackenby, Agol, and D. Thurston \cite{LackenbyBound}. First we recall some terminology from \cite{LackenbyBound}. 
\begin{definition}
A \textit{twist region} of a diagram $D$ is either a connected collection of bigon regions of $D$ arranged in a row, which is maximal in the sense that it is not part of a longer row of bigons, or a single crossing adjacent to no bigon regions. The \textit{twist number} of a diagram $D$ is the number of twist regions in the diagram. A diagram is \textit{twist reduced} whenever a simple closed curve in the diagram intersects the link projection
transversely in four points disjoint from the crossings, and two of these points are adjacent to some crossing, and the remaining two points are adjacent to some other crossing, then this curve bounds a subdiagram that consists of a (possibly empty) collection of bigons arranged in a row between these two crossings.
\end{definition}

\begin{theorem}[\cite{LackenbyBound}]\label{thm:LackenbyBound}
Let $K$ be an alternating hyperbolic link with $t$ twist regions in a prime alternating diagram. Then $$\emph{vol}(K) < 10 v_4(t-1)$$ where $v_4 \approx 1.01494$ is the volume of a regular ideal tetrahedron.
\end{theorem}

\noindent The upper bound of Theorem \ref{thm:LackenbyBound} can be improved in the case of Montesinos links using work of Futer, Kalfagianni, and Purcell.
\begin{theorem}[{\cite{Guts}}]\label{thm:montesinos}
Let $K$ be a hyperbolic Montesinos link. Then $\emph{vol}(K) < 2 v_8 t$ where $t$ is the number of twists in some diagram of $K$ and $v_8 \approx 3.66386237$ is the volume of a regular ideal hyperbolic octahedron.
\end{theorem}
Note that while the statement of Theorem \ref{thm:montesinos} in \cite{Guts} requires the link to have at least three positive tangles and at least three negative tangles, this condition was only necessary to prove the lower bound stated in that theorem.

\subsection{Determinants of Links}
The determinant of a link $K$ is defined by $\det(K) = |\Delta_K(-1)|$ where $\Delta_K(t)$ is the Alexander polynomial. It is well-known when $K$ is alternating, the determinant is equal to the number of spanning trees of any of the checkerboard graphs for $K$ (see for example \cite[Lemma 3.14]{StoimenowGraphsDeterminants}). Recall that the checkerboard graphs for $K$ are constructed as follows. Take a reduced, alternating diagram $D$ of $K$ and then checkerboard color $D$. Create a graph $G$ by having one vertex per shaded region of the checkerboard coloring of $D$, and connect vertices with one edge per crossing of $D$ connecting the corresponding shaded regions. See Figure \ref{fig:Checkerboard} for an example. \\

We recall two methods that one may use to compute the number of spanning trees of a graph. First we present the Matrix Tree Theorem proved by Kirchoff in 1847. One can find a modern proof in \cite{MatrixTreeTheorem}.

\begin{figure}
\includegraphics[scale=1]{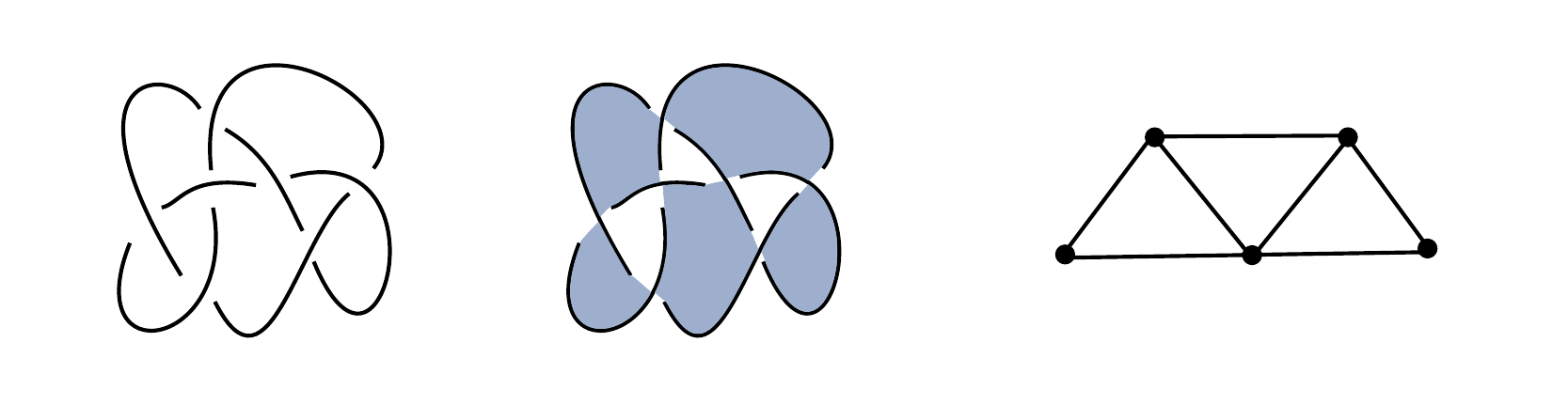}
\caption{\textsc{Left}: The $7_7$ knot. \textsc{Center}: Checkerboard coloring. \textsc{Right}: The checkerboard graph associated to the shaded faces.}\label{fig:Checkerboard}
\end{figure}

\begin{theorem}[Matrix Tree Theorem]\label{thm:MatrixTree}
Let $G$ be a graph and let $v_1, \hdots, v_n$ be the vertices of $G$. Let $\alpha(i,j)$ be the number of edges with endpoints on both of the vertices $v_i$ and $v_j$. Define $L$ to be the matrix (known as the Laplacian) where the $(i,j)$ entry $\ell_{ij}$ of $L$ is given by
\begin{equation}
\ell_{ij} = \left\{ \begin{array}{cc} \deg(v_i)-2\alpha(i,i) & \text{ if } j = i\\
-\alpha(i,j) & \text{ if } j \neq i \end{array} \right.
\end{equation}
Then $\tau(G)$ is given by the determinant of any of the $(n-1) \times (n-1)$ minors of $L$.
\end{theorem}

For the next lemma, we introduce some notation. Let $G$ be a graph and $e$ an edge of the graph. We define $G-e$ to be the graph obtained by removing the edge $e$ from $G$. We define $G/e$ to be the graph obtained by contracting the edge $e$ and identifying the endpoints of $e$ to a single vertex as shown in Figure \ref{fig:collapse}. With this notation, we recall the following well-known result.

\begin{figure}
\includegraphics[scale=1]{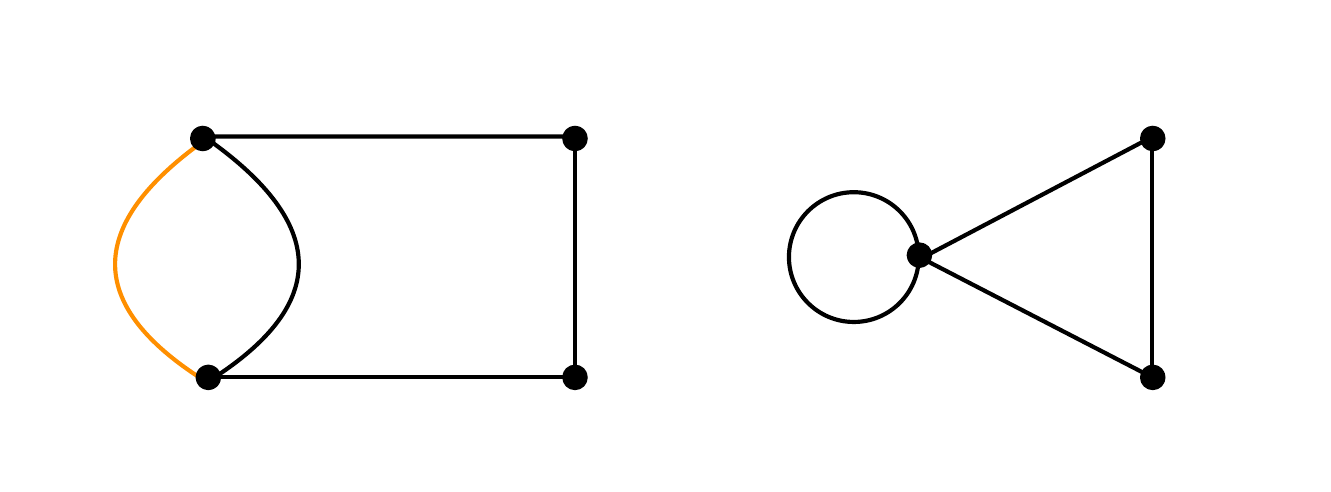}
\caption{\textsc{Left}: Initial graph. \textsc{Right}: Result of collapsing the orange edge.}\label{fig:collapse}
\end{figure}

\begin{lemma}\label{lem:ContractionDeletion}
Let $G$ be a graph. Then $\tau(G) = \tau(G-e) + \tau(G/e)$.
\end{lemma}

Using spanning trees, Stoimenow \cite{StoimenowGraphsDeterminants} was able to give a lower bound on the determinant of an alternating knot.

\begin{theorem}[{\cite[Theorem 4.3]{StoimenowGraphsDeterminants}}]\label{thm:Stoimenow}
Let $t$ be the number of twist regions in a twist-reduced alternating diagram $D$ of a link $K$. Then $$\det(K) \geq 2 \cdot \gamma^{t-1}$$ where $\gamma \approx 1.4253$ is the unique positive real number satisfying $\gamma^{-5} + 2\gamma^{-4} + \gamma^{-3} -1 = 0$.
\end{theorem}

\section{Two Bridge Links}\label{sec:two_bridge}
It is known that any 2-bridge link has an alternating projection of one of the forms shown in Figure \ref{fig:TwoBridge}. We will denote 2-bridge links by $R(a_1, a_2,\hdots,a_n)$ where the sequence $a_1, a_2,\hdots,a_n$ denotes the number of half-twists in each crossing region. Examples of $R(3,3,2)$ and $R(3,2,2,3)$ are provided in Figure \ref{fig:TwoBridge}. We begin the proof that Conjecture \ref{conj:det_vol} holds for 2-bridge links by studying the determinant of a 2-bridge link. Kauffman and Lopes \cite{KauffmanRational} gave the following recursive method of calculating the determinant of rational links.
\begin{figure}
\includegraphics[scale=.7]{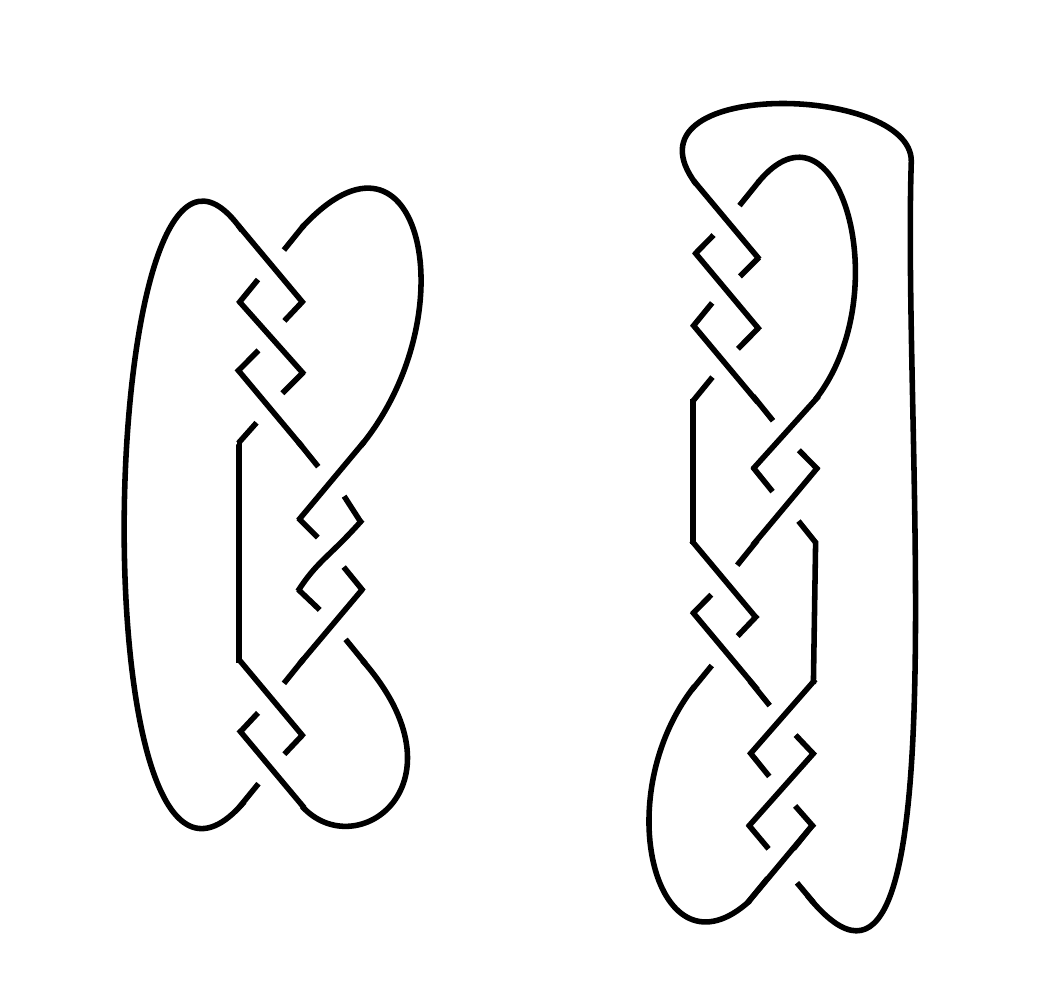}
\caption{\textsc{Left}: The knot $R(3,3,2)$. \textsc{Right}: The link $R(3,2,2,3)$.}\label{fig:TwoBridge}
\end{figure}

\begin{theorem}[\cite{KauffmanRational}]\label{thm:DetRecurrence}
Let $K = R(a_1, a_2, \hdots, a_n)$ be a two-bridge link. Then $det(K) = T(n)$ where $T(n)$ is defined by the recursion
\begin{equation}\label{eqn:detrec}
\left\{\begin{array}{ccl}
T(0) & = & 1\\
T(1) & = & a_1\\
T(k+1) & = & a_{k+1}T(k) + T(k-1)
\end{array}
\right.
\end{equation}
\end{theorem}

\noindent It is interesting to note that when $a_k = 1$ for all $k$, then the recursion yields the Fibonacci sequence. We now introduce some notation that will aid the exposition. Define 
\begin{equation}
V(a_1, \hdots, a_n) = \prod_{i = 1}^n \frac{(a_i + 2)}{2}
\end{equation}
Let $K = R(a_1, \hdots, a_n)$. Note that by Corollary \ref{cor:AdamsVolumeBound} we have 
\begin{equation} 
\vol[K] < 2 \pi \log\left(\frac{(a_1 + 1)(a_n + 1)}{4}\prod_{i = 2}^{n-1} \frac{(a_i + 2)}{2}\right) < 2 \pi \log(V(a_1, \hdots, a_n)) 
\end{equation}

We will obtain a lower bound on $T(n)$ from the the recurrence of Theorem \ref{thm:DetRecurrence} and then show that it exceeds $V(a_1, \hdots, a_n)$. The most problematic cases for obtaining a lower bound on $T(n)$ are when $a_i = 1$ for many values of $i$. The following lemma allows us to reduce to the case where $a_1, \hdots, a_n$ contains no long sequences of consecutive ones.

\begin{lemma}\label{lem:reductionfactor}
Let $K = R(a_1, a_2, \hdots, a_n)$ and let $T(i)$ be the recurrence (\ref{eqn:detrec}). Fix $k \geq 2$. Let $$K' = R(a_1,\hdots,a_{k-1},a_{k+m},\hdots, a_n)$$ and define another recurrence $\widehat{T}(i)$ by
\begin{equation}
\widehat{T}(i) = \left\{ 
\begin{array}{lrl} 
T(i) & \text { if } & i < k\\
a_{i+1}\widehat{T}(i-1) + \widehat{T}(i-2) & \text{ if } & i \geq k
\end{array}
\right.
\end{equation}
Then the following are true:
\begin{enumerate}[label=(\alph*)]
\item If $T(k) > \frac{3}{2}T(k-1)$ and $T(k-1) > \frac{3}{2}T(k-2)$, then $T(n) > \frac{3}{2}\widehat{T}(n-1)$ and $$\det(R(a_1, \hdots, a_n)) > \frac{3}{2} \det(R(a_1, \hdots, a_{k-1}, a_{k + 1}, \hdots, a_n))$$ \label{part:reductionfactora}
\item Suppose $k \geq 4$ and $a_{k-2} = a_{k-1} = a_{k} = \hdots = a_{k+m-1} = 1$ for some $m \geq 1$. Then
$$\det(K) > \left(\frac{3}{2} \right)^m \det(K') $$ \label{part:reductionfactorb}
\item Suppose $k \geq 4$ and $a_{k-2} = a_{k-1} = a_{k} = \hdots = a_{k+m-1} = 1$ for some $m \geq 1$. If $$2 \pi \log (V(a_1, \hdots, a_{k-1}, a_{k + m}, \hdots, a_n)) \leq 2 \pi \log(\det(K'))$$ then  $\emph{vol}(K) < 2 \pi \log(\det(K))$. \label{part:reductionfactorc}
\end{enumerate} 
\end{lemma}
\begin{proof}
To begin the proof of part \ref{part:reductionfactora} by proving the following claim: 
\begin{equation}\label{eqn:redaclaim}
T(k + i + 1) > (3/2)\widehat{T}(k+i) \text{ for all } i \geq -1
\end{equation} 
The case where $i = -1$ holds since 
\begin{equation*}
T(k) > \dfrac{3}{2}T(k-1) = \dfrac{3}{2}\widehat{T}(k-1)
\end{equation*}
To prove the case where $i = 0$, note that
\begin{align*}
T(k + 1) & = a_{k+1} T(k) + T(k-1)\\
& > \frac{3}{2} [a_{k+1}T(k-1) + T(k-2)]\\
& = \frac{3}{2}\widehat{T}(k)
\end{align*}
We now proceed by induction. Assume that $T(k + i + 1) > \frac{3}{2}\widehat{T}(k+i)$ and $T(k + i) > \frac{3}{2}\widehat{T}(k + i -1)$. Then 
\begin{align*}
T( k + i + 2 ) & = a_{k + i + 2} T(k + i + 1) + T(k + i)\\
& > \frac{3}{2} [ a_{k + i + 2} \widehat{T}(k + i) + \widehat{T}(k + i -1)]\\
& = \frac{3}{2} \widehat{T}(k + i + 1)
\end{align*}
This proves (\ref{eqn:redaclaim}). Observe that $T(n) = \det(R(a_1, \hdots, a_n))$ and $\widehat{T}(n-1) = \det(R(a_1, \hdots, a_{k-1},a_{k+1}, \hdots, a_n))$, thus completing the proof of part \ref{part:reductionfactora}.\\

\noindent We prove part \ref{part:reductionfactorb} by induction on $m$. Note that 
\begin{align*}
T(k-1) & = T(k-2) + T(k-3)\\
& = 2T(k-3) + T(k-4) && \text{since } T(k-2) = T(k-3) + T(k-4)\\
& \geq \frac{3}{2}T(k-3) + \frac{3}{2}T(k-4)&& \text{since } T(k-3)\geq T(k-4)\\
& = \frac{3}{2}T(k-2)
\end{align*}
Similarly, $T(k) \geq \frac{3}{2} T(k-1)$. Then by part \ref{part:reductionfactora}
\begin{equation}
\det(R(a_1, \hdots, a_n)) = T(n) > \frac{3}{2} \widehat{T}(n-1) = \frac{3}{2}\det (R(a_1, \hdots, a_{k-1}, a_{k+1}, \hdots, a_n))
\end{equation}
This proves the case where $m = 1$. Assume that 
\begin{equation}\label{induct1}
\det(K) > \left(\dfrac{3}{2}\right)^{m-1} \det(R(a_1,\hdots, a_{k-1},a_{k+m-1},\hdots, a_n)
\end{equation} 
Since $a_{k-2} = a_{k-1} = a_{k+m-1} = 1$ we may use part \ref{part:reductionfactora} to obtain
\begin{align*}\label{induct2}
\det(K) & >\left(\dfrac{3}{2}\right)^{m-1}\det(R(a_1, \hdots, a_{k-1}, a_{k+m-1}, a_{k+m}, \hdots, a_n)) && \text{by (\ref{induct1})}\\
& > \left(\dfrac{3}{2}\right)^m\det(R(a_1, \hdots, a_{k-1}, a_{k+m}, \hdots, a_n)) && \text{by part \ref{part:reductionfactora}}
\end{align*}
which finishes the proof of part \ref{part:reductionfactorb}. 
\noindent To prove part \ref{part:reductionfactorc}, observe that if 
\begin{equation}\label{eqn:assume}
2 \pi \log(V(a_1, \hdots, a_{k-1}, a_{k + m}, \hdots a_n)) < \det(R(a_1, \hdots, a_{k-1}, a_{k+m}, \hdots, a_n))
\end{equation}
then 
\begin{align*}
\vol[K] & < 2 \pi \log (V(a_1, \hdots, a_n))\\
& = 2 \pi \log \left(\left(\frac{3}{2}\right)^m V(a_1, \hdots, a_{k-1}, a_{k + m}, \hdots, a_n)\right)\\
& < 2 \pi \log \left(\left(\frac{3}{2}\right)^m \det(R(a_1, \hdots, a_{k-1}, a_{k+m}, \hdots, a_n))\right) && \text{by (\ref{eqn:assume})}\\
& < 2 \pi \log(\det(K)) && \text{by part \ref{part:reductionfactorb}}
\end{align*} 
\end{proof}

Using Lemma \ref{lem:reductionfactor} we can reduce to the case where we do not have $a_{k-2} = a_{k-1} = a_{k} = 1$ for any $k \geq 4$, \textit{i.e.} there is no subsequence of three or more consecutive ones. Next we will prove Lemma \ref{lem:decomposition} which empowers us to bound $\det(R(a_1,\hdots, a_n))$ by breaking up the sequence $a_1,\hdots,a_n$ into shorter subsequences.

\begin{lemma}\label{lem:decomposition}
Let $1 \leq k \leq n-1$. Then 
$$\det R(a_1, \hdots,a_n) > \det R(a_1, \hdots, a_k) \det(R(a_{k+1}, \hdots, a_n))$$
\end{lemma}
\begin{proof}
Let $T(i)$ be the recursion defined in Theorem \ref{thm:DetRecurrence}. Define the following recursive sequences $T'(i)$ and $T''(i)$:
\begin{align}\label{eqn:seq1}
&\left\{
\begin{array}{lll}
T'(0) & = & 1 \\
T'(1) & = & a_{k+1}\\
T'(i) & = & a_{k + i}T'(i-1) + T'(i-2)
\end{array} \right.
\\
\label{eqn:seq2}
&\left\{
\begin{array}{lll}
T''(0) & = & 1 \\
T''(1) & = & a_{k+2}\\
T''(i) & = & a_{k + i + 1}T''(i-1) + T''(i-2)
\end{array} \right.
\end{align}
Note that 
\begin{equation}\label{eqn:note}
T'(n-k) = \det(R(a_{k+1}, a_{k+2}, \hdots, a_n))
\end{equation}
We will show that
\begin{equation}\label{eqn:claim}
T(k + m) = T'(m)T(k) + T''(m-1)T(k-1)
\end{equation}
for $m \geq 1$. We proceed by induction on $m$. When $m = 1$ we have
\begin{align}
T(k + 1) &= a_{k+1}T(k) + T(k-1) &&\text{by (\ref{eqn:detrec})}\\
& = T'(1)T(k) + T''(0)T(k-1) && \text{since $T'(1) = a_{k+1}$ and $T''(0) = 1$ by definition } \label{eqn:m=1}
\end{align} 
When $m = 2$ we have
\begin{align*}
T(k + 2) & = a_{k + 2}T(k+1) + T(k) && \text{by (\ref{eqn:detrec})}\\
& = a_{k+2}T'(1)T(k) + a_{k+2}T''(0)T(k-1) + T'(0)T(k) && \text{by (\ref{eqn:m=1}) and $T'(0) = 1$}\\
& = [a_{k+2}T'(1) + T'(0)]T(k) + a_{k+2}T(k-1) && \text{since $T''(0) = 1$}\\
& = T'(2)T(k) + T''(1)T(k-1) && \text{by (\ref{eqn:seq1}) and $T''(1) = a_{k+2}$ }
\end{align*}
Now assume that 
\begin{equation}\label{eqn:inducthyp}
T(k + m) = T'(m)T(k) + T''(m-1)T(k-1)
\end{equation} for $m \geq 2$. Then
\begin{align*}
T(k + m + 1) & = a_{k + m + 1}T(k+m) + T(k + m - 1)
\end{align*}
Which by applying (\ref{eqn:inducthyp}) to $T(k + m)$ and $T(k + m -1)$ becomes  
\begin{align}\label{eqn:twoliner}
T(k + m + 1) & = a_{k + m + 1}T'(m)T(k) + a_{k + m + 1}T''(m-1)T(k-1)\\
\notag &\quad\quad + T'(m-1)T(k) + T''(m-2)T(k-1)
\end{align}
By collecting like terms (\ref{eqn:twoliner}) simplifies to
\begin{align}
T(k + m + 1) & = [a_{k+m+1}T'(m) + T'(m-1)]T(k) \label{eqn:collect}\\
\notag &\quad \quad + [a_{k+m+1}T''(m-1) + T''(m-2)]T(k-1)
\end{align}
By (\ref{eqn:seq1}) we have that 
\begin{equation}
a_{k+m+1}T'(m) + T'(m-1) = T'(m+1)\label{eqn:fact1}
\end{equation}
and by (\ref{eqn:seq2}) we also know that 
\begin{equation}
a_{k+m+1}T''(m-1) + T''(m-2) = T''(m)\label{eqn:fact2}
\end{equation} 
Combining (\ref{eqn:collect}), (\ref{eqn:fact1}), and (\ref{eqn:fact2})  we see that  
\begin{align*}
T(k + m + 1) = T'(m+1)T(k) + T''(m)T(k-1)
\end{align*}
Finally, we observe that
\begin{align*}
\det(R(a_1, \hdots, a_n)) & = T(n)\\
& = T(k + n - k)\\
& = T'(n-k)T(k) + T''(n - k -1) T(k-1) && \text{by (\ref{eqn:claim})}\\
& > T'(n-k)T(k)\\
& = \det(R(a_{k + 1}, \hdots, a_n))\det(R(a_1, \hdots, a_k)) && \text{by (\ref{eqn:note})}
\end{align*}
\end{proof} 

Using Lemmata \ref{lem:reductionfactor} and \ref{lem:decomposition} we will break up the sequence $a_1,\hdots, a_n$ into smaller subsequences which will have one of the eleven special types listed in the following lemma.

\begin{lemma}\label{lem:sixcases}
Let $a_1, \hdots, a_n$ be a sequence of one of the following eleven types:
\begin{enumerate}
\item $a_1$ where $a_1 \geq 2$ 
\item $1, a_2$ where $a_2 \geq 2$
\item $a_1, 1$ where $a_1 \geq 2$
\item $1, 1, a_3$ where $a_3 \geq 2$
\item $a_1, 1, 1$ where $a_1 \geq 2$
\item $1, 1, 1, a_4$ where $a_4 \geq 2$
\item $1, a_2, 1, 1$ where $a_2 \geq 2$
\item $1, 1, a_3, 1$ where $a_3 \geq 2$
\item $1, 1, a_3, 1, 1$ where $a_3 \geq 2$
\item $1, a_2, 1, a_4, 1$ where $a_2 \geq 2$ and $a_4 \geq 2$
\item $1, 1, a_3, 1, a_5, 1$ where $a_3 \geq 2$ and $a_5 \geq 2$
\end{enumerate}
Let $T(i)$ be the recurrence of Theorem \ref{thm:DetRecurrence}. Then $V(a_1, \hdots, a_n) \leq T(n)$. 
\end{lemma}
\begin{proof}
For type (1), one readily obtains that ${a}_1 + 2 \leq 2 {a}_1$ implying $$V(a_1) = \frac{{a}_1+2}{2} \leq \frac{2{a}_1}{2} =  T(1)$$
For type (2), we see that $T(2) = a_2 + 1$ and $$V(1, a_2) = \frac{3(a_2 + 2)}{4}$$ When $a_2 \geq 2$, one readily obtains $V(1, a_2) < T(2)$.\\

\noindent For type (3), we have $T(2) = a_1 + 1$ and $$V(a_1, 1) = \frac{3(a_1 + 2)}{4}$$ and the proof proceeds in a similar manner to type (2).\\

\noindent For type (4), $T(3) = 2 a_3 + 1$ and $$V(1, 1, a_3) = \frac{9(a_3+2)}{8}$$ When $a_3 \geq 2$ one readily obtains $$\frac{9(a_3+2)}{8} \leq 2 a_3 + 1$$\\
For type (5), we have $T(3) = 2a_1 + 1$ and $$V(a_1, 1, 1) = \frac{9(a_1 + 2)}{8}$$ and the proof proceeds similarly to type (4).\\

\noindent For type (6), we have $T(4) = 3a_4 + 2$ and $$V(1, 1, 1, a_4) = \frac{27(a_4 + 2)}{16}$$ When $a_4 \geq 2$, one may show that $$  \frac{27(a_4 + 2)}{16} \leq 3a_4 + 2$$\\
For type (7), we have $T(4) = 2a_2 + 3$ while $$V(1,a_2,1,1) = \frac{27}{16}(a_2 + 2)$$ Then $$T(4) - V(1,a_2,1,1) = \frac{5}{16}a_2 - \frac{3}{8} \geq 0$$ since $a_2 \geq 2$. \\

\noindent For type (8), we have $T(4) = 2a_3 + 3$ and $$V(1,1, a_3, 1) = \frac{27}{16}(a_3+2)$$ The proof is now similar to type (7).\\

\noindent For type (9), we have $T(5) = 4a_3 + 4$ and $$V(1,1,a_3,1,1) = \frac{81}{32}(a_3 + 2)$$ Then $$T(5) - V(1,1,a_3,1,1) = \frac{47}{32}a_3 - \frac{17}{16} \geq 0$$ since $a_3 \geq 2$. \\

\noindent For type (10), we have $T(5) = a_2a_4 + 2a_2 + 2a_4 + 3$ and $$V(1,a_2,1,a_4,1) = \frac{27}{32}(a_2a_4 + 2a_2 + 2a_4 + 4)$$ Then $$T(5) - V(1,a_2,1,a_4,1) = \frac{5}{32}(a_2a_4 + 2a_2 + 2a_4) - \frac{3}{8} \geq 0$$ since $a_2 \geq 2$ and $a_4 \geq 2$.\\

\noindent For type (11), we have $T(6) = 2a_3a_5 + 4a_3 + 3a_5 + 4$ an $$V(1, 1, a_3, 1, a_5, 1) = \frac{81}{64}(a_3a_5 + 2a_3 +2a_5 + 4)$$ Then $$T(6) - V(1,1,a_3,1,a_5,1) = \frac{47}{64}a_3a_5 + \frac{47}{32}a_3 + \frac{15}{32}a_5 - \frac{17}{16} \geq 0$$ since $a_3 \geq 2$ and $a_5 \geq 2$.
\end{proof}

\noindent We are now prepared to present the proof of Theorem \ref{thm:rationaldetvol}.

\begin{theorem}\label{thm:rationaldetvol}
Let $K$ be the 2-bridge link $R(a_1, a_2, \hdots, a_n)$. Then $\emph{vol}(K) < 2 \pi \log \det(K)$.
\end{theorem}

\begin{proof}
We consider three cases: 
\begin{enumerate}
\item $n = 3, a_1 = 1, a_2 \geq 2, a_3 = 1$ \label{case:0}
\item $a_1=a_2 = \hdots = a_n = 1$ \label{case:one}
\item $a_i \neq 1$ for some $i$ and not case \ref{case:0}\label{case:2}
\end{enumerate}
\emph{Case \ref{case:0}}:\\
We can calculate that $\det(K) = a_2 + 2$. Using Corollary \ref{cor:AdamsVolumeBound} we see that 
\begin{equation}
\vol[K] < 2 \pi \log \left(\frac{a_2 + 2}{2} \right) < 2 \pi \log(a_2 + 2) = 2 \pi \log(\det(K))
\end{equation}

\noindent For the remaining cases, it suffices to show that $V(a_1,\hdots,a_n) \leq \det(K)$.\\

\noindent\emph{Case \ref{case:one}}:\\
If $n = 1$, $2$, or 3 then $K$ is $R(1)$, $R(1,1)$, or $R(1,1,1)$ respectively, none of which is hyperbolic. If $n = 4$ then $\det(K) = T(4) = 5$. Corollary \ref{cor:AdamsVolumeBound} implies that $$\vol[K] < 2 \pi \log\left(\frac{2\cdot 3 \cdot 3 \cdot 2}{2^4}\right) = 2 \pi \log \left(\frac{9}{4}\right) < 2 \pi \log(5) = 2 \pi \log (\det(K))$$ 
If $n = 5$ then $\det(K) = T(5) = 8$ and while $V(1,1,1,1,1) = 243/32 < 8$.
If $n \geq 6$ then we use Lemma \ref{lem:reductionfactor} part \ref{part:reductionfactorc}. We can let $k = 6$ and then the link $K'$ defined in Lemma \ref{lem:reductionfactor} part \ref{part:reductionfactorc} will be $R(1,1,1,1,1)$. The result now follows from the case $n = 5$ above. \\

\noindent \emph{Case \ref{case:2}}:\\
We may use Lemma \ref{lem:reductionfactor} part \ref{part:reductionfactorc} to assume that we do not have $a_{k-2} = a_{k-1} = a_k = 1$ for any $k \geq 4$, \textit{i.e.} $a_1,\hdots,a_n$ has no subsequences of three or more consecutive ones except possibly $a_1 = a_2 = a_3 = 1$. Let $m$ be the cardinality $|\{a_k \in \{a_i\}_{i = 1}^n : a_k \geq 2\}|$. We will partition the sequence $\{a_1, \hdots, a_n\} = \{b_1^{(1)}, \hdots, b_{n_1}^{(1)}, b_1^{(2)}, \hdots b_{n_2}^{(2)}, \hdots, b_1^{(m)}, \hdots b_{n_m}^{(m)}\}$ into subsequences of the types found in Lemma \ref{lem:sixcases} according to one of the cases below.\\

\noindent \emph{Case \ref{case:2}a}: $a_1 \geq 2$\\
Let $b_1^{(k)}$ be the $k$th element of the sequence $a_1, \hdots, a_n$ that is greater than or equal to 2. Then each subsequence $b_1^{(k)},\hdots, b_{n_k}^{(k)}$ is either type 1, 3, or 5 from Lemma \ref{lem:sixcases}.\\

\noindent \emph{Case \ref{case:2}b}: $a_1 = 1$ and $a_n \geq 2$\\
Let $b_{n_{k}}^{(k)}$ be the $k$th element of $a_1,\hdots,a_n$ that is greater than or equal to 2. Then each $b_1^{(k)},\hdots, b_{n_k}^{(k)}$ is either of type 1, 2, 4, or 6 in \ref{lem:sixcases}.\\

\noindent \emph{Case \ref{case:2}c}: $a_1 = a_{n-1} = a_n = 1$\\
For $1 \leq k \leq m-1$ let $b_{n_{k}}^{(k)}$ be the $k$th element of $a_1,\hdots,a_n$ that is greater than or equal to 2. Then each $b_1^{(k)},\hdots, b_{n_k}^{(k)}$ is either of type 1, 2, 4, or 6 in Lemma \ref{lem:sixcases}. Let $b_1^{(m)}, \hdots, b_{n_m}^{(m)}$ be the remaining elements of the sequence $a_1,\hdots, a_n$. Then $b_1^{(m)}, \hdots, b_{n_m}^{(m)}$ is either of type 7 or 9 from Lemma \ref{lem:sixcases}.\\

\noindent \emph{Case \ref{case:2}d}: $a_1 = a_n = 1$ and $a_{n-1} \neq 1$\\
For $1 \leq k \leq m-2$ let $b_{n_{k}}^{(k)}$ be the $k$th element of $a_1,\hdots,a_n$ that is greater than or equal to 2. Then each $b_1^{(k)},\hdots, b_{n_k}^{(k)}$ is either of type 1, 2, 4, or 6 in Lemma \ref{lem:sixcases}. Let $b_1^{(m)}, \hdots, b_{n_{m}}^{(m)}$ be the remaining elements of the sequence $a_1,\hdots, a_n$. Since case \ref{case:2} excludes cases \ref{case:0} and \ref{case:one}, $\{a_1, a_2, a_3\} \neq \{1, a_2, 1\}$. Therefore $b_1^{(m)}, \hdots, b_{n_{m}}^{(m-1)}$ is either of type 8, 10, or 11 from Lemma \ref{lem:sixcases}. For notational purposes, take $b_1^{(m-1)}, \hdots, b_{n_{m-1}}^{(m-1)}$ to be the empty sequence.\\

Now that we have partitioned the sequence $\{a_1, \hdots, a_n\} = \{b_1^{(1)}, \hdots, b_{n_1}^{(1)}, b_1^{(2)}, \hdots b_{n_2}^{(2)}, \hdots, b_1^{(m)}, \hdots b_{n_m}^{(m)}\}$ into subsequences of the types found in Lemma \ref{lem:sixcases}, we observe that
\begin{align*}
\det(K) & > \det(R( b_1^{(1)}, \hdots, b_{n_1}^{(1)}))\det (R(b_1^{(2)}, \hdots, b_{n_2}^{(2)})) \hdots \det (R(b_1^{(m)}, \hdots, b_{n_m}^{(m)})) && \text{by Lemma \ref{lem:decomposition}} \\
& \geq V(b_1^{(1)}, \hdots, b_{n_1}^{(1)}) V(b_1^{(2)}, \hdots, b_{n_2}^{(2)}) \hdots V(b_1^{(m)}, \hdots, b_{n_m}^{(m)}) && \text{by Lemma \ref{lem:sixcases}}\\
& = V(a_1, \hdots, a_n)
\end{align*}

\end{proof}

The proof of Theorem \ref{thm:rationaldetvol} also verified the following fact which will be used in Section \ref{sec:3_braids}.

\begin{corollary}\label{cor:detandV}
Let $K = R(a_1,\hdots, a_n)$ be a 2-bridge link. If $K \neq R(1,1,1,1)$, $R(1,a_2,1)$, $R(1,1)$ or $R(1)$ for any $a_2 \geq 1$ then
$$\det(R(a_1,\hdots,a_n)) \geq V(a_1,\hdots,a_n) $$
\end{corollary}

\section{Alternating Braids}\label{sec:3_braids}

We will show in this section that Conjecture \ref{conj:det_vol} holds for alternating 3-braids and for an infinite family of 4-braids. The former fact will rely on the result of Theorem \ref{thm:rationaldetvol}, while the latter fact will be proved by bounding the hyperbolic volume and explicitly computing the determinant.

\subsection{3-Braids} Let $B(a_1, b_1, \hdots, a_n, b_n)$ denote the alternating 3-braid with $a_1$ positive crossings in the first twist region, $b_1$ negative crossings in the second twist region, and so on. See for example Figure \ref{fig:RationalBraidCheckerboard}. Note that up to reflection, this considers all alternating 3-braids, and that up to isomorphism all alternating 3-braids have an even number of twist regions. We state the main theorem for this section.

\begin{theorem}\label{thm:3braids}
If $K$ is an alternating 3-braid then $\emph{vol}(K) < 2 \pi \log \det(K)$.
\end{theorem}
The proof of Theorem \ref{thm:3braids} will follow immediately from from Lemmata \ref{lem:ThreeBraidEnds} and \ref{lem:all1}. Note that by using Corollary \ref{cor:AdamsVolumeBound} it is sufficient to show that $V(a_1, b_1, \hdots, a_n, b_n) \leq \det(B(a_1, b_1, \hdots, a_n, b_n))$.

\begin{figure}
\includegraphics[scale=1]{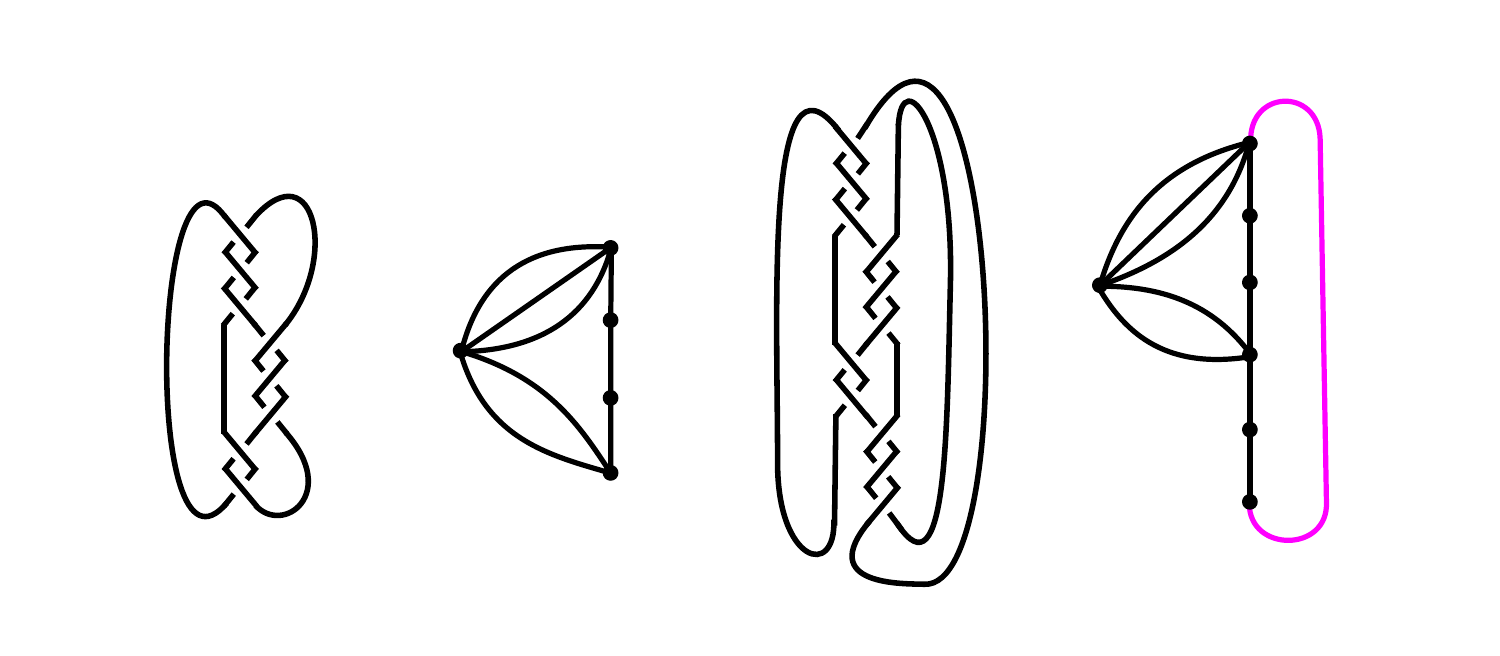}
\caption{\textsc{Left}: The two-bridge knot $R(3,3,2)$. \textsc{Center Left}: The checkerboard graph for $R(3,3,2)$. \textsc{Center Right}: The closed alternating 3-braid $B(3,3,2,3)$. \textsc{Right}: The checkerboard graph for $B(3,3,2,3)$. Note that the graph that results from deleting the highlighted edge has the same number of spanning trees as the checkerboard graph for $R(3,3,2)$. }\label{fig:RationalBraidCheckerboard}
\end{figure}

\begin{lemma}\label{lem:ThreeBraidEnds}
Let $K = B(a_1, b_1, \hdots, a_n, b_n)$. If $a_i \neq 1$ or $b_i \neq 1$ for some $i$ then $\emph{vol}(K) < 2 \pi \log \det(K)$.
\end{lemma}
\begin{proof}
Suppose $b_i \neq 1$. Let $\sigma_1$ and $\sigma_2$ be generators of the 3-braid, where $\sigma_i$ denotes a positive half-twist of the $i$th and $(i+1)$st strands. Then $K$ is the closure of $\sigma_1^{a_1}\sigma_2^{-b_1}\hdots\sigma_1^{a_n} \sigma_2^{-b_n}$. Let $\alpha = (\sigma_2^{b_n}\sigma_1^{-a_n}\sigma_2^{b_{n-1}}\sigma_1^{a_{n-1}}\hdots \sigma_2^{b_{i+1}}\sigma_1^{-a_{i+1}})$. Then 
\begin{equation}\label{eqn:braidconj}
\alpha^{-1}(\sigma_1^{a_1}\sigma_2^{-b_1}\hdots\sigma_1^{a_n} \sigma_2^{-b_n})\alpha = \sigma_1^{a_{i-1}}\sigma_2^{b_{i-1}}\hdots \sigma_1^{a_n}\sigma_2^{-b_n} \sigma_1^{a_1}\sigma_2^{-b_1}\hdots \sigma_1^{a_i}\sigma_1^{-b_i}
\end{equation} 
The braid closure of the right hand side of (\ref{eqn:braidconj}) corresponds to the three-braid $$K' = B(a_{i-1}, b_{i-1}, \hdots, a_n, b_n, a_1,b_1, \hdots, a_{i}, b_{i})$$
so $K$ is equivalent to $K'$. Therefore we may assume in this case that $b_n \neq 1$. We will now reduce the problem to the case of 2-bridge links, and the result will follow from Theorem \ref{thm:rationaldetvol}. Observe that one of the checkerboard graphs for $B(a_1, b_1, \hdots, a_n, b_n)$ and $R(a_1, b_1, \hdots, b_{n-1}, a_n)$ will be of the form shown in Figure \ref{fig:RationalBraidCheckerboard}. Apply Lemma \ref{lem:ContractionDeletion} by contracting and deleting the highlighted right edge in the far right of Figure \ref{fig:RationalBraidCheckerboard}. The result of deleting the edge yields a graph with the same number of spanning trees as one of the checkerboard graphs of $R(a_1,b_1, \hdots, b_{n-1}, a_n)$. If $b_n \geq 2$ then contraction of the highlighted edge is the checkerboard graph of $B(a_1, b_1, \hdots, a_n, b_n - 1)$. Further, contracting the highlighted edge of a checkerboard graph of $B(a_1,b_1,\hdots, a_n, 1)$ yields a graph isomorphic to a checkerboard graph of $B(a_1 + a_n, b_1, a_2, b_2, \hdots, a_{n-1}, b_{n-1})$. Therefore we see inductively that 
\begin{align}
\notag\det(B(a_1, b_1, \hdots, a_n, b_n)) & = \det(R(a_1, b_1, \hdots, b_{n-1}, a_n)) + \det(B(a_1, b_1, \hdots, a_n, b_n - 1))\\
\notag & = 2\det(R(a_1, b_1, \hdots, b_{n-1}, a_n)) + \det(B(a_1, b_1, \hdots, a_n, b_n - 2))\\
\notag &\,\,\,\vdots \\
\notag & = (b_n-1)\det(R(a_1, b_1, \hdots, b_{n-1}, a_n)) + \det(B(a_1, b_1, \hdots, a_n, 1))\\
& = b_n\det(R(a_1, b_1, \hdots, b_{n-1}, a_n)) + \det(B(a_1+a_n, b_1, \hdots, a_{n-1}, b_{n - 1}))\label{eqn:calculatebraid}\\
& \geq b_n\det(R(a_1, b_1, \hdots, b_{n-1}, a_n))\label{eqn:braidineq}
\end{align} 
Suppose that $\{a_1,b_1,\hdots, a_n, b_n\} \neq \{1, b_1, 1, b_2\}$ or $\{1, b_1\}$. Then since $b_n \geq 2$, we have $b_n \geq (b_n + 2)/2$ and we can apply Corollary \ref{cor:detandV}:
\begin{align}
\notag \det(B(a_1, b_1, \hdots, a_n, b_n)) & \geq b_n\det(R(a_1, b_1, \hdots, b_{n-1}, a_n)) && \text{by \ref{eqn:braidineq}}\\
& \geq b_n V(a_1, b_1, \hdots, b_{n-1}, a_n) && \text{by Corollary \ref{cor:detandV}} \label{eqn:usecor}\\
\notag & = b_n \dfrac{\prod_{i = 1}^n(a_i + 2) \prod_{i = 1}^{n-1}(b_i + 2)}{2^{2n - 1}}\\
\notag & \geq \dfrac{\prod_{i = 1}^n(a_i + 2) \prod_{i = 1}^{n}(b_i + 2)}{2^{2n}}\\
\notag & = V(a_1, b_1, \hdots, a_n, b_n)
\end{align}
If $\{a_1,b_1,\hdots, a_n, b_n\} = \{1, b_1, 1, b_2\}$ or $\{1, b_1\}$ then we cannot apply Corollary \ref{cor:detandV} to obtain (\ref{eqn:usecor}). If $\{a_1,b_1,\hdots, a_n, b_n\}  = \{1, b_1\}$ then $B(1, b_1)$ is seen to be the $(2, b_1)$-torus link which is not hyperbolic. If $\{a_1,b_1,\hdots, a_n, b_n\} = \{1, b_1, 1, b_2\}$ then using (\ref{eqn:calculatebraid}) we see that 
\begin{align}
\det(B(1, b_1, 1, b_2)) & = b_2 \det(R(1,b_1, 1)) + \det(B(2, b_1))\\
& = b_2(b_1 + 2) + 2b_1
\end{align} 
On the other hand
\begin{equation}
V(1, b_1, 1, b_2) = \frac{9}{16}(b_1b_2 + 2b_1 + 2b_2 + 4)
\end{equation}
Since $b_2 \geq 2$ we see that
\begin{align*}
\det(B(1,b_1,1,b_2)) - V(1,b_1,1,b_2) & = \frac{7}{16}(b_1 b_2 + 2 b_1 + 2b_2) - \frac{36}{16}\\
& \geq \frac{56}{16} - \frac{36}{16}\\&> 0
\end{align*}
If $b_i = 1$ for all $i$, then $a_i \neq 1$ for some $i$. Then $K$ is equivalent to the link $$B(a_{i}, b_{i}, \hdots, a_n, b_n, a_1,b_1, \hdots, a_{i-1}, b_{i-1})$$ and therefore we assume that $a_1 \neq 1$. Then we can repeat the argument above by considering the other checkerboard surface (\textit{i.e.} considering the checkerboard graph obtained from the white regions instead of the shaded regions).
\end{proof}

\begin{lemma}\label{lem:all1}
Let $K = B(1, 1, \hdots, 1, 1)$ where there are $2n$ copies of $1$. Then $\emph{vol}(K) < 2 \pi \log \det(K)$.
\end{lemma}
\begin{proof}
We will use the Matrix Tree Theorem, Theorem \ref{thm:MatrixTree}, to compute the number of spanning trees of the checkerboard graph, and hence the determinant of $K$. If $n \geq 3$, the associated checkerboard graph has Laplacian
\begin{equation}
L = \begin{bmatrix}
 n 	   & -1		 & -1 & -1 & -1 & -1 & \hdots & -1 & -1 & -1\\
-1 	   &  3 	 & -1 &  0 &  0 & 0  & \hdots &  0 &  0 & -1\\
-1 	   & -1 	 &  3 & -1 &  0 & 0  & \hdots &  0 &  0 &  0 \\
-1 	   &  0 	 & -1 &  3 & -1 & 0  & \hdots &  0 &  0 &  0\\
-1 	   &  0 	 &  0 & -1 &  3 & -1 & 		   &  0 &  0 &  0\\
\vdots & \vdots & \vdots & \vdots & \vdots & \vdots &  & \vdots & \vdots\\
-1 & 0  & 0 & 0 & 0 & 0 & \hdots &  -1 &   3 & -1\\
-1 & -1 & 0 & 0 & 0 & 0 & \hdots &   0 &  -1 &  3
\end{bmatrix}
\end{equation}
where $L$ is an $(n+1) \times (n +1)$ matrix.
Let $L'$ be the minor of $L$ obtained by eliminating the first row and first column. Then it is known (see for example \cite{MolinariTridiagonal}) that 
\begin{equation}
\det(L') = -2 + \text{tr} \left(\prod_{i = 1}^{n} \begin{bmatrix} 3 & -1 \\ 1 & 0 \end{bmatrix}\right)
\end{equation}
By diagonalizing, we can compute an explicit formula:
\begin{align*}
\det(L') & = -2 + \text{tr} \left( \mathlarger{\mathlarger{\prod}}_{i = 1}^n
\begin{bmatrix} \dfrac{3 + \sqrt{5}}{2} & \dfrac{3 - \sqrt{5}}{2} \\ \\ 1 & 1\end{bmatrix}
\begin{bmatrix} \dfrac{3 + \sqrt{5}}{2} & 0 \\ 0 & \dfrac{3 - \sqrt{5}}{2}\end{bmatrix}
\begin{bmatrix} \dfrac{3 + \sqrt{5}}{2} & \dfrac{3 - \sqrt{5}}{2} \\ \\ 1 & 1\end{bmatrix} \inv\right)\\
& = -2 + \text{tr} \left(
\begin{bmatrix} \dfrac{3 + \sqrt{5}}{2} & \dfrac{3 - \sqrt{5}}{2} \\ \\ 1 & 1\end{bmatrix}
\begin{bmatrix} \left(\dfrac{3 + \sqrt{5}}{2}\right)^n & 0 \\ 0 & \left(\dfrac{3 - \sqrt{5}}{2}\right)^n\end{bmatrix}
\begin{bmatrix} \dfrac{3 + \sqrt{5}}{2} & \dfrac{3 - \sqrt{5}}{2} \\ \\ 1 & 1\end{bmatrix} \inv\right)\\
& = -2 + \left(\dfrac{3 + \sqrt{5}}{2}\right)^n + \left(\dfrac{3 - \sqrt{5}}{2}\right)^n
\end{align*}
The volume of $K$ is bounded above by 
\begin{equation}
V(1, 1, \hdots, 1, 1) = \left( \frac{3}{2} \right)^{2n} = \left( \frac{9}{4} \right)^n 
\end{equation}
It is straightforward to check that 
\begin{equation}
-2 + \left(\dfrac{3 + \sqrt{5}}{2}\right)^n + \left(\dfrac{3 - \sqrt{5}}{2}\right)^n > \left( \frac{9}{4} \right)^n \quad \text{ if } n \geq 3
\end{equation}
If $ n = 1$, then $K$ is not hyperbolic. If $n =2$ then $K$ is the figure-eight knot, which is a 2-bridge knot and therefore satisfies Conjecture \ref{conj:det_vol} by Theorem \ref{thm:rationaldetvol}.
\end{proof}

\subsection{A Family of 4-braids}
Let $\sigma_1, \sigma_2, \sigma_3$ be the generators of the 4-braid, where $\sigma_i$ denotes a positive half-twist of the $i$th and $(i+1)$st strands. Let $W_n$ be the closure of $(\sigma_1\sigma_3\sigma_2\inv)^n$. We note that these links correspond to the weaving links $W(4,n)$ of \cite{CKP} and \cite{CKP2}. A checkerboard graph associated with $W_n$ is the maximal planar lantern graph $\mathcal{E}_{n+2}$ on $(n+2)$ vertices as shown in Figure \ref{fig:FourBraid}. Work of Modabish, Lotfi, and El Marraki \cite{MobadishPlanar} shows that 
\begin{equation}\label{eqn:PlanarDet}
\tau(\mathcal{E}_{n+2}) = \dfrac{n}{2 +2\sqrt{3}} \left[ (2 + \sqrt{3})^{n} - (2 - \sqrt{3})^{n}\right] \leq \dfrac{n+2}{2\sqrt{3}}(2+\sqrt{3})^{n}
\end{equation}
On the other hand, Corollary \ref{cor:AdamsVolumeBound} shows that
\begin{equation}\label{eqn:PlanarVol}
\vol[W_n] \leq 2 \pi \log \left( \dfrac{3^{2n}4^{n}}{2^{3n}}\right) = 2 \pi \log \left[\left(\dfrac{9}{2}\right)^{n} \right]
\end{equation}
Since $2 + \sqrt{3} < 9/2$ it follows from equations (\ref{eqn:PlanarDet}) and (\ref{eqn:PlanarVol}) that the volume bound of Corollary \ref{cor:AdamsVolumeBound} is insufficient to prove Conjecture \ref{conj:det_vol} for these links. However, one may instead use Theorem \ref{thm:AdamsUpperBound} to find that
\begin{align*}
\exp\left(\dfrac{\vol[W_n]}{2 \pi}\right)  & \leq \exp\left[ 2n\, \vol[B_3]+n\,\vol[B_4]\right]\\
& \leq (3.418677233748620053022)^n
\end{align*}
This bound may also be obtained from \cite[Theorem 1.1]{CKP2}.
On the other hand, equation (\ref{eqn:PlanarDet}) and the fact that $(2-\sqrt{3})^{n} < \frac{1}{2}(2 + \sqrt{3})^{n}$ for $ n \geq 1$ together imply that
\begin{equation}
\tau(\mathcal{E}_{n+2}) = \dfrac{n+2}{2 \sqrt{3}} \left[ (2 + \sqrt{3})^{n} - (2 - \sqrt{3})^{n}\right] \geq \dfrac{n+2}{4 \sqrt{3}}(2 +\sqrt{3})^{n}
\end{equation}
It is straightforward to show that $$(3.418677)^{n}<\dfrac{n+2}{4 \sqrt{3}}(2 +\sqrt{3})^{n} \text{ for $n \geq 4$}$$ so Conjecture \ref{conj:det_vol} holds for all $W_n$ with $n \geq 4$. Note that the case $n \leq 3$ has been verified in \cite{CKP}.

\begin{figure}
\includegraphics[scale=.75]{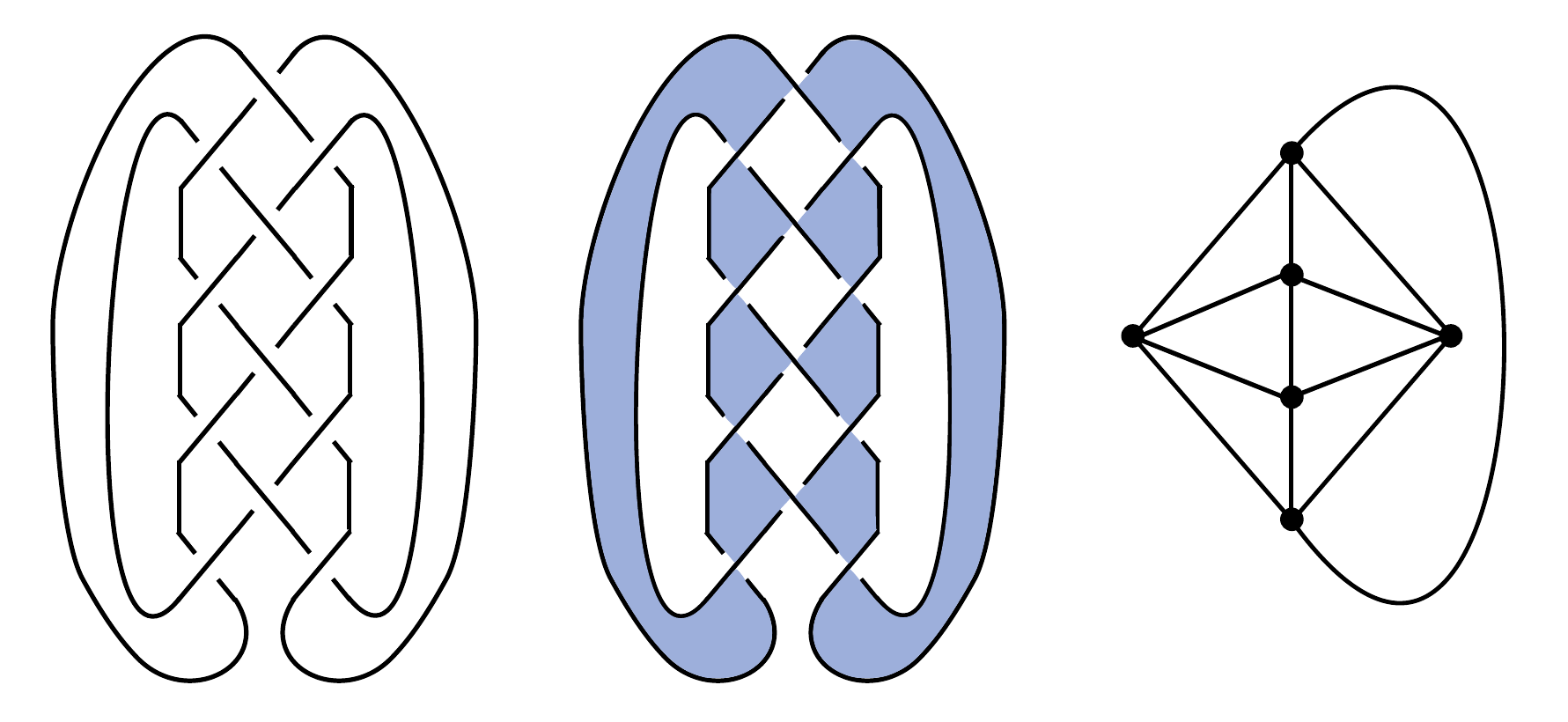}
\caption{\textsc{Left}: The link $W_n$. \textsc{Center}: The checkerboard coloring of $W_4$. \textsc{Right}: The graph $\mathcal{E}_6$ corresponding to the white checkerboard surface.}\label{fig:FourBraid}
\end{figure}

One can use this method to find many more infinite families of links for which Conjecture \ref{conj:det_vol} holds. Given a planar graph $G$, one may create an alternating link $K$ for which $G$ is the checkerboard graph of $K$. This is done by replacing each edge with a crossing and connecting ends of crossings so that each vertex is on the shaded part of the checkerboard surface. One can then calculate the volume estimates and then if the number of spanning trees of the graph is known test whether the conjecture holds. This method works for the wheel, fan, crystal, star-flower graphs of \cite{MobadishStarFlower} and \cite{MobadishPlanar} as well as the grid graphs and triangulated grid-graphs of \cite{MobadishFamilies}.

\section{Highly Twisted Knots}\label{sec:pretzel}

We consider the situation where a link has a twist region with many crossings. It is known by \cite{LackenbyBound} that the volume of an alternating link is bounded by the number of twist regions in the diagram. Therefore, increasing the number of crossings in a twist region of an alternating hyperbolic link has a bounded effect on the hyperbolic volume. On the other hand, the number of spanning trees in the checkerboard graph will increase by adding crossings to a twist region. It follows that highly twisted links must satisfy Conjecture \ref{conj:det_vol}. We quantify this in the following theorem.

\begin{theorem}\label{thm:high_twist}
Let $K$ be an alternating hyperbolic link with a reduced alternating diagram having $t$ twist regions and $c$ crossings. If 
\begin{equation}\label{eqn:twist_bound}
c \geq t + \xi^{t-1} - 2\gamma^{t-1}
\end{equation}
where $\gamma \approx 1.4253$ is as described in Theorem \ref{thm:Stoimenow} and $\xi = e^{5 v_{4}/\pi} \approx 5.0296$, then $\emph{vol}(K) < 2 \pi \log(\det(K))$.
\end{theorem}
\begin{proof}
Let $a_1, \hdots, a_t$ be the crossing numbers of the $t$ twist regions of $K$. Let $G(x_1,\hdots, x_t)$ be the checkerboard graph of the link obtained by placing $x_i$ crossings in the $i$th twist region of $K$. Since a checkerboard graph has the same number of spanning trees as its  dual, we may assume that the first twist region of $K$ corresponds to a path on $a_1$ vertices in $G(a_1,\hdots, a_t)$. By Lemma \ref{lem:ContractionDeletion} we then obtain
\begin{align*}
\tau(G(a_1, a_2, \hdots, a_t)) & = \tau(G(a_1-1,a_2\hdots, a_t)) + \tau(G(0, a_2, \hdots, a_t))\\
& = \tau(G(1,a_2,\hdots, a_t)) + (a_1 - 1) \tau(G(0,\hdots, a_t))\\
& \geq \tau(G(1,a_2,\hdots, a_t)) + (a_1 - 1)\\
& \geq \tau(G(1,1,\hdots, 1)) + \sum_{i = 1}^t(a_i - 1)\\
& = \tau(G(1, 1, \hdots, 1)) + c - t
\end{align*}
Theorem \ref{thm:Stoimenow} then implies that
\begin{equation}
\det(K) = \tau(G(a_1,\hdots, a_t)) \geq 2 \gamma^{t-1} + c - t
\end{equation}
By Theorem \ref{thm:LackenbyBound} we know that $\vol[K] \leq 10 v_{4}(t-1)$. It is then straightforward to check that if (\ref{eqn:twist_bound}) holds then 
\begin{equation}
\vol[K] < 10v_{4}(t-1) \leq 2 \pi \log(2 \gamma^{t-1} + c - t) \leq 2 \pi \log(\det(K))
\end{equation} 
\end{proof}

\begin{corollary}\label{cor:Montesinos}
Let $K$ be an alternating hyperbolic Montesinos link with $t$ twist regions and $c$ crossings. If
\begin{equation}
c \geq t + \zeta^t - 2 \gamma^{t-1}
\end{equation}
where $\zeta = e^{v_{8}/\pi} \approx 3.2099$ then $\emph{vol}(K) \leq 2 \pi \log(\det(K))$.
\end{corollary}
\begin{proof}
The proof is the same as for Theorem \ref{thm:high_twist}, except we replace the upper bound on volume with the bound $2 v_{8}t$ of Theorem \ref{thm:montesinos}.
\end{proof}

\subsection{Application to Pretzel Knots}
We give an application of Theorem \ref{thm:high_twist} and Corollary \ref{cor:Montesinos} to alternating pretzel links. We begin by calculating the determinant of an alternating pretzel knot.
\begin{proposition}\label{prop:PretzelDet}
 Let $P(a_1, a_2, \hdots, a_n)$ be the alternating pretzel knot having $a_1, a_2, \hdots, a_n$ crossings in the first, second, and so on to the $n$th twist region. Then
 \begin{equation}
 \det(P(a_1,a_2,\hdots, a_n)) = \sum_{i = 1}^n\prod_{j \neq i} a_j
 \end{equation}
\end{proposition}
\begin{proof}
We will use Lemma \ref{lem:ContractionDeletion}. An example checkerboard graph for $P(a_1, a_2, \hdots, a_n)$ is given on the left of Figure \ref{fig:Pretzel}. Deleting an edge in the $n$th twist region produces a graph with the same number of spanning trees as the checkerboard graph of $P(a_1, \hdots, a_{n-1})$. For example one may delete the red edge in Figure \ref{fig:Pretzel}. On the other hand, if $a_n \geq 2$, then contracting that edge results in the checkerboard graph of $P(a_1, a_2 \hdots, a_n -1)$. If $a_n = 1$ then the resulting graph is the join of $(n-1)$ cycles. For example, contracting the blue edge in Figure \ref{fig:Pretzel} produces the graph on the right of Figure \ref{fig:Pretzel}) which has $a_1a_2\hdots a_{n-1}$ spanning trees. Therefore by Lemma \ref{lem:ContractionDeletion} we see that \begin{equation}
\det(P(a_1,\hdots, a_n)) = a_1\hdots a_{n-1} + a_n \det(P(a_1, a_2,\hdots, a_{n-1}))
\end{equation} 
Applying the above method to $P(a_1, \hdots, a_{n-1})$ \textit{et cetera} we obtain the desired result.
\end{proof}
\begin{figure}
\includegraphics[scale=.75]{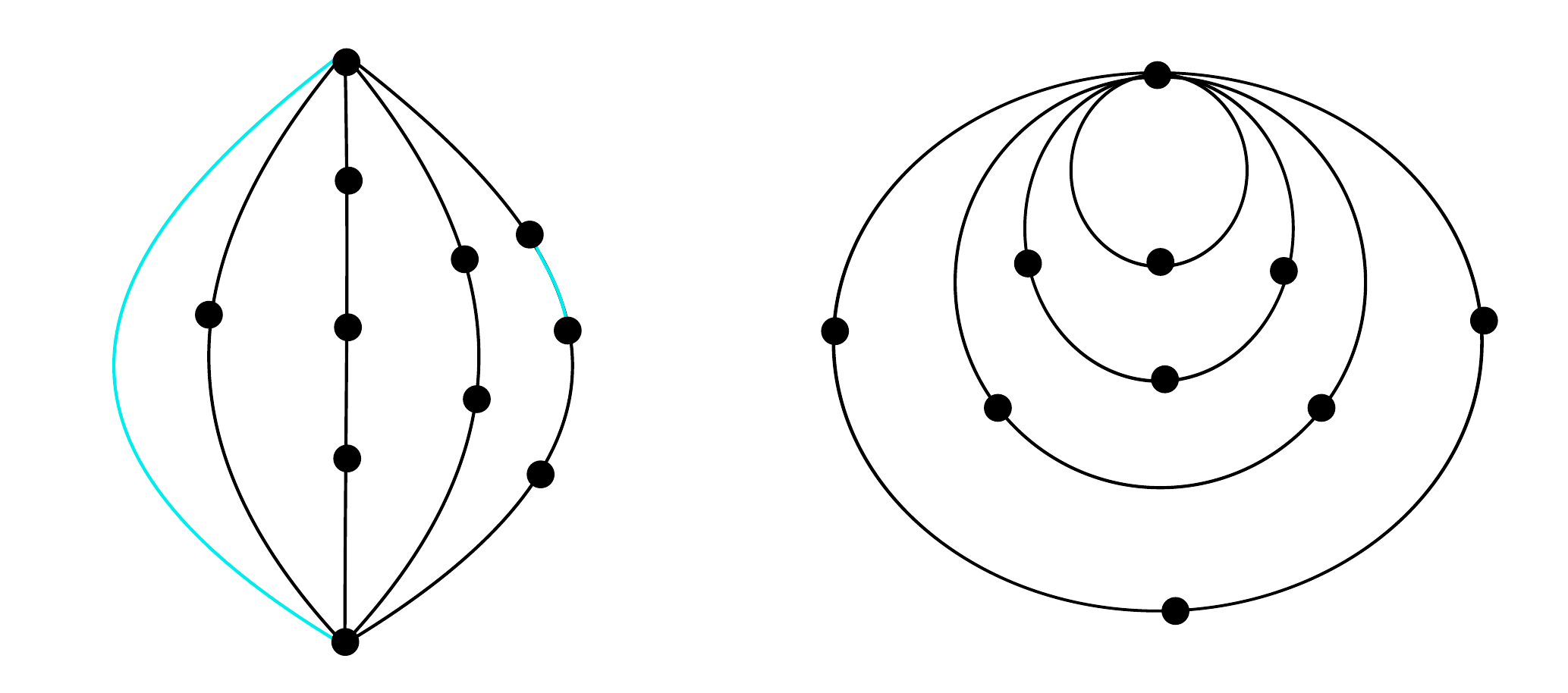}
\caption{\textsc{Left}: The checkerboard graph of the Pretzel knot $P(1,2,4,3,4)$. Deletion of the red edge produces a graph with the same number of spanning trees as the checkerboard graph of $P(1,2,4,3)$. Contraction of the cyan edge yields the graph on the right. \textsc{Right}: Result of contracting the cyan edge of the graph on the left.}
\label{fig:Pretzel}
\end{figure}

Given a fixed number $t$ of twist regions and a pretzel link $K$ with $t$ twist regions, Corollary \ref{cor:Montesinos} can be used to say that if $K$ has more than t +$ \zeta^t - 2 \gamma^{t-1} $ crossings in any twist region, then it satisfies Conjecture \ref{conj:det_vol}. Therefore for a given $t$, there are only finitely many links which may fail to satisfy Conjecture \ref{conj:det_vol}. We may enumerate these links, use Theorem \ref{thm:AdamsUpperBound} to compute an upper bound on volume, and check that this upper bound is less than the determinant. Using this method, we have shown with computer assistance that Conjecture \ref{conj:det_vol} holds for all alternating pretzel links with no more than 13 twist regions.
\newpage

\bibliographystyle{abbrv}

\bibliography{DetVolConjecture}

\begin{thebibliography}{10}

\bibitem{AdamsBipyramids}
C.~Adams.
\newblock Bipyramids and bounds on volumes of hyperbolic links.
\newblock {\em arXiv:1511.02372v1}, 2015.

\bibitem{AdamsEtAl}
C.~Adams, A.~Kastner, A.~Calderon, X.~Jiang, G.~Kehne, N.~Mayer, and M.~Smith.
\newblock Volume and determinant densities of hyperbolic rational links.
\newblock {\em J. Knot Theory Ramifications}, 26(1):1750002, 13, 2017.

\bibitem{MatrixTreeTheorem}
S.~Chaiken.
\newblock A combinatorial proof of the all minors matrix tree theorem.
\newblock {\em SIAM J. Algebraic Discrete Methods}, 3(3):319--329, 1982.

\bibitem{CKP}
A.~Champanerkar, I.~Kofman, and J.~S. Purcell.
\newblock Geometrically and diagramatically maximal knots.
\newblock {\em arXiv:1411.7915}, 2015.

\bibitem{CKP2}
A.~Champanerkar, I.~Kofman, and J.~S. Purcell.
\newblock Volume bounds for weaving knots.
\newblock {\em Algebr. Geom. Topol.}, 16(6):3301--3323, 2016.

\bibitem{SnapPy}
M.~Culler, N.~M. Dunfield, M.~Goerner, and J.~R. Weeks.
\newblock Snappy, a computer program for studying the geometry and topology of
  3-manifolds.
\newblock {\em Available at http://snappy.computop.org}.

\bibitem{DunfieldOnline}
N.~Dunfield.
\newblock {\em http://www.math.uiuc.edu/~nmd/preprints/misc/dylan/index.html}.

\bibitem{Guts}
D.~Futer, E.~Kalfagianni, and J.~Purcell.
\newblock {\em Guts of surfaces and the colored {J}ones polynomial}, volume
  2069 of {\em Lecture Notes in Mathematics}.
\newblock Springer, Heidelberg, 2013.

\bibitem{Knotscape}
J.~Hoste and M.~Thistlethwaite.
\newblock Knotscape 1.01.
\newblock {\em Available at http://www.math.utk.edu/~morwen/knotscape.html}.

\bibitem{KauffmanRational}
L.~H. Kauffman and P.~Lopes.
\newblock Determinants of rational knots.
\newblock {\em Discrete Math. Theor. Comput. Sci.}, 11(2):111--122, 2009.

\bibitem{LackenbyBound}
M.~Lackenby.
\newblock The volume of hyperbolic alternating link complements.
\newblock {\em Proc. London Math. Soc. (3)}, 88(1):204--224, 2004.
\newblock With an appendix by Ian Agol and Dylan Thurston.

\bibitem{MobadishFamilies}
A.~Modabish and M.~El~Marraki.
\newblock The number of spanning trees of certain families of planar maps.
\newblock {\em Appl. Math. Sci. (Ruse)}, 5(17-20):883--898, 2011.

\bibitem{MobadishStarFlower}
A.~Modabish and M.~El~Marraki.
\newblock Counting the number of spanning trees in the star flower planar map.
\newblock {\em Appl. Math. Sci. (Ruse)}, 6(49-52):2411--2418, 2012.

\bibitem{MobadishPlanar}
A.~Modabish, D.~Lotfi, and M.~El~Marraki.
\newblock Formulas for the number of spanning trees in a maximal planar map.
\newblock {\em Appl. Math. Sci. (Ruse)}, 5(61-64):3147--3159, 2011.

\bibitem{MolinariTridiagonal}
L.~G. Molinari.
\newblock Determinants of block tridiagonal matrices.
\newblock {\em Linear Algebra Appl.}, 429(8-9):2221--2226, 2008.

\bibitem{StoimenowGraphsDeterminants}
A.~Stoimenow.
\newblock Graphs, determinants of knots and hyperbolic volume.
\newblock {\em Pacific J. Math.}, 232(2):423--451, 2007.

\end{thebibliography}

\end{document}